\documentclass[12pt]{article}

\usepackage{amsmath}
\usepackage{amssymb}
\usepackage{tikz}
\tikzstyle{vertex}=[circle, draw, inner sep=0pt, minimum size=6pt]

\usepackage{amsthm} 

\definecolor{claret}{RGB}{141,40,56}

\newtheorem{theorem}{Theorem}[section]
\newtheorem{corollary}[theorem]{Corollary}
\newtheorem{lemma}[theorem]{Lemma}
\newtheorem{definition}[theorem]{Definition}

\usepackage[linesnumbered,lined,boxed,commentsnumbered]{algorithm2e}

\title {An introduction to the deduction number}
\author{Andrea Burgess\thanks{Department of Mathematics and Statistics, University of New Brunswick, Saint John, NB  E2L 4L5, Canada.  \texttt{andrea.burgess@unb.ca}}, Danny Dyer\thanks{Department of Mathematics and Statistics, Memorial University of Newfoundland, St.\ John's, NL  A1C 5S7, Canada.  \texttt{dyer@mun.ca}} and Mozhgan Farahani\thanks{Department of Mathematics and Statistics, Memorial University of Newfoundland, St.\ John's, NL  A1C 5S7.  \texttt{mfarahani@mun.ca}}}

\begin{document}

\maketitle

\begin{abstract}
The deduction game is a variation of the game of cops and robber on graphs in which searchers must capture an invisible evader in at most one move. Searchers know each others' initial locations, but can only communicate if they are on the same vertex. Thus, searchers must deduce other searchers' movement and move accordingly. We introduce the deduction number and study it for various classes of graphs. We provide upper bounds for the deduction number of the Cartesian product of graphs. 
\end{abstract}

\section{Introduction}

The game of cops and robber belongs to the larger class of pursuit-evasion games on graphs. A pursuit-evasion game is a two-player game in which pursuers attempt to capture evaders, while the evaders try to prevent this capture in a graph. The pursuers may be represented by cops, searchers, or even zombies. The evaders might be robbers, lost people, or survivors. Pursuit-evasion games can be applied to unmanned aerial vehicles \cite{H}, differential games \cite{IS}, robotics \cite{Chung}, control theory \cite{Sh}, and so on.

Pursuit-evasion on a graph was first introduced by Breisch \cite{Bre}, though much of the deeper mathematical work was first done by Parsons \cite{Parsons}. In this searching game, a fixed number of searchers who know the structure of a cave try to find a lost child. The goal becomes determining the minimum number of searchers to guarantee the lost child is found. Since then, other variants of the game have been studied. 

These games  vary based on the capabilities of the players, the information available to each player, the environment in which the game occurs, the strategies allowed to the 
pursing agents or searchers to track evaders, and so forth. However, the main feature in all of them is that the searchers are asked to find an evader who wishes to avoid the searchers. The phrase ``to capture the evader'' typically means that some searcher occupies the same vertex as the evader. How the agents move may vary in different models. In one game, agents might move from a vertex to an adjacent vertex \cite{Tosic2}. In another game, agents may be able to jump from a vertex to another non-adjacent vertex \cite{Abbas}. Searchers may know the evader's location \cite {Winkler} or may not \cite{Parsons}.

The deduction game considers the cops and robber model with some extra restrictions for the cops, which we term searchers. Imagine there is a city with some residential areas connected by streets. There might be one or more evaders hidden in these areas, and the searchers aim to capture all these evaders, but the searchers lack information about the number of evaders and the evaders' locations. (Alternatively, we may assume that if there is a possibility that the evader escapes capture, then they do.) Searchers also have no means to communicate when they are in different areas. The question becomes: what strategy should searchers use to guarantee capturing the evader(s), and further, how few searchers are required for this capture? 

In Section 2, we introduce the deduction model and compare it with other well-known models. In Section 3, we introduce some bounds involving elementary graph parameters, and then use these to find the deduction number of various graph families in Section 4. In Section 5, we tackle the problem of trees. Finally, in Section 6, we end by posing some new questions. For elementary graph theoretical definitions, we follow \cite{West}.


\section{The deduction model}

In the initial setup, searchers are placed on vertices of a graph $G$; we refer to the arrangement of searchers on $G$ as a {\em layout}. 
Note that in a layout, a vertex may contain more than one searcher.  In a layout, a vertex is {\em unoccupied} if it contains no searcher, and {\em occupied} if it contains at least one searcher.  The searchers know the initial layout and the structure of the graph; however, searchers located on different vertices cannot communicate to coordinate their movement.  The searchers' aim is to capture an invisible intruder located on a vertex of $G$.

We next describe the rules governing the searchers' movement.  We say that a vertex is {\em protected} if it is occupied in the initial layout or if a searcher has moved there at some point, and {\em unprotected} otherwise.  A vertex is {\em fireable} if the number of searchers on that vertex is greater than or equal to the number of its unprotected neighbours.  Initially each searcher is considered {\em mobile}; once a searcher has moved once, it becomes {\em immobile}, i.e.\ cannot move again.  At each {\em stage} of the game, for every fireable vertex $v$ containing mobile searchers, one searcher moves from $v$ to each unprotected neighbour of $v$;  these vertices are referred to as the {\em target} vertices of $v$.  Note that if the number of searchers on $v$ exceeds the number of unprotected neighbours, some searchers will remain on $v$.  The game proceeds in successive stages until no new searchers move.
We say that the searchers win if all vertices become protected; otherwise the intruder wins.  If the searchers win, we call the layout $L$ {\em successful}.  

Since each searcher may move only once, we may consider all searchers as moving at the same time in practice, and think of the stages of movement described previously as a process of searchers {\em deducing} what their movement will be.
However, by an abuse of terminology we speak of a searcher moving to a target vertex at the time they deduce the location of their target.

As an example, consider the cycle $C_8$ with the layout illustrated in Figure~\ref{8-cycle}.  In the first stage, only $v_1$ and $v_2$ are fireable, and the searchers on these vertices move to $v_8$ and $v_3$.  Thus, in the second stage $v_4$ becomes fireable, and the searcher on this vertex moves to $v_5$.  Finally, in the third stage, the searcher on $v_6$ moves to $v_7$.  At this point, all vertices are protected, so the layout is successful.
\begin{figure}[ht]
\centering
\begin{tikzpicture}[x=1cm,y=1cm,scale=1]
\foreach \x in {0,...,7}{
    \draw (45*\x:2) -- (45*\x+45:2);
}
\foreach \x in {0,...,7}{
    \draw[fill=white] (45*\x:2) circle (4pt);
}
\draw (90:2) node[above, inner sep=6pt]{$v_1$};
\draw (90:2) node[below, inner sep=6pt]{$S$};
\draw (45:2) node[above right]{$v_2$};
\draw(45:2) node[below left]{$S$};
\draw (0:2) node[right, inner sep=6pt]{$v_3$};
\draw (-45:2) node[below right]{$v_4$};
\draw (-45:2) node[above left]{$S$};
\draw (-90:2) node[below, inner sep=6pt]{$v_5$};
\draw (-135:2) node[below left]{$v_6$};
\draw (-135:2) node[above right]{$S$};
\draw (180:2) node[left, inner sep=6pt]{$v_7$};
\draw (135:2) node[above left]{$v_8$};
\end{tikzpicture}
\caption{A successful layout on $C_8$ \label{8-cycle}}
\end{figure}
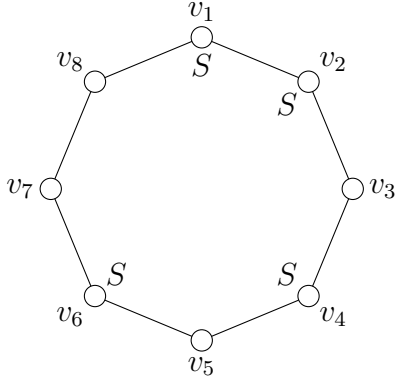

At a given stage of the game, if there are more unprotected neighbours than the number of searchers on a vertex, they fail to deduce their movement and are {\it flummoxed}. Searchers may remain flummoxed for a fixed collection of deductions, until their vertex becomes fireable. Note that it is possible for a searcher to remain in the same position throughout the game (for instance, if all their neighbours are protected or they remain flummoxed throughout); we refer to such searchers as {\it motionless}.

We are interested in determining the minimum number of searchers required in a successful layout on a given graph $G$.  We call this number the \emph{deduction number} of $G$ and denote it by $d(G)$. 
 
Suppose that there is a graph $G$ with a successful layout $L$. Then the layout $L^{\ast}$ is a {\it terminal layout} for $L$ in $G$ if, after the motion of all searchers that need to move for protection of their adjacent vertices, $L$ produces $L^{\ast}$. 
Given a layout $L$, its terminal layout is unique. However, it is not true in general that $(L^*)^*=L$.
For example, consider the path $P_{3}=v_1v_2v_3$. One successful layout $L$ contains one searcher on $v_{1}$ and another searcher on $v_{2}$. The searcher on $v_{2}$ moves to $v_{3}$, producing the terminal layout $L^{\ast}$, which contains one searcher on vertex $v_{1}$ and one searcher on vertex $v_{3}$.  If we play the game with initial layout $L^{\ast}$, we see that $(L^*)^*$ has both searchers on $v_2$.
 
The deduction game bears some relation to certain variants of the game of cops and robber, which was first introduced independently by Nowakowski and Winkler~\cite{Winkler} and Quilliot~\cite{Quilliot}, and extended to include more than one cop by Aigner and Fromme~\cite{Aigner}.  In this game, the goal of the cops is to capture a robber in a graph $G$, where players have perfect information, may communicate, and have no restriction on the number of movements they can make. The {\em cop number $c(G)$} is the minimum number of cops required to guarantee capture of the robber in $G$ in a finite number of moves.  When the robber is invisible to the cops except at the time of capture, we refer to the {\em zero-visibility} version of the game and speak of the {\em zero-visibility cop number $c_0(G)$}.  This variant was first introduced by To\v{s}i\'{c}~\cite{Tosic2}; for more information, see~\cite{DDTY1, DDTY2, Tang}. Time-constained versions of cops and robber, in which the cops must capture the robber within a given number of turns, were considered in~\cite{time}.  In particular, $1$-$c(G)$ and $1$-$c_0(G)$ are the number of cops required to capture the robber in at most one turn if the robber is visible and invisible, respectively.  In~\cite{time}, it is shown that $1$-$c(G) = \gamma(G)$, the domination number of $G$, and $1$-$c_0(G) = \beta'(G)$, the size of a minimum edge cover of $G$.

In contrast to deduction, all of these games allow cops to communicate to coordinate their movements; thus a cop may move from a vertex to an adjacent vertex even if the cop's other neighbours are not all protected.  Nevertheless, the searchers' movement in a successful game of deduction yield a strategy for cops to win each of these other games, leading to the following bounds.  

\begin{theorem}
In any graph $G$, $c(G) \leq 1$-$c(G) \leq d(G)$ and $c_0(G) \leq 1$-$c_0(G) \leq d(G)$.
\end{theorem}

Notably, however, the difference between $d(G)$ and $1$-$c(G)$ or $1$-$c_0(G)$ can be arbitrarily large.  For instance, $1$-$c(K_{n})=1$ and $1$-$c_0(K_n)=\lceil n/2\rceil$, while Theorem~\ref{Th_4} will show that $d(K_{n})=n-1$.

The deduction game is also reminiscent of graph processes such as chip-firing~\cite{chip} and brushing~\cite{brush}, in which chips or brushes are transferred from vertices to their neighbours according to a condition regarding their neighbourhood.  In the chip-firing game introduced in~\cite{chip}, initially chips are placed on vertices of a graph.  A vertex is considered fireable if its number of chips exceeds its degree, and in each stage a fireable vertex $v$ is chosen, and one chip is moved from $v$ to each of its neighbours.  Other variations of chip-firing games have also been studied, see for instance~\cite{BitarGoles, DLMN}.  However, one characteristic shared by these games is that chips may be moved to vertices which already contain chips; by contrast, in deduction searchers only move to unoccupied vertices.  In the graph cleaning model introduced in~\cite{brush}, brushes take the place of searchers/chips, and move from one vertex to an adjacent vertex if the number of brushes on a vertex exceeds the number of adjacent dirty edges.  Unlike the searchers in deduction, however, brushes may move repeatedly, and the brush number of a graph can differ from its deduction number by an arbitrary amount; for example, the brush number of a path on $n$ vertices is $1$, while in Theorem~\ref{CO_3}, we will see that $d(P_n) = \lceil \frac{n}{2}\rceil$.  Further connections between deduction and other graph searching models are considered in~\cite{BDOXY}.

\section{Bounds and properties}
We begin our discussion of the deduction number by considering bounds. One trivial arrangement that is a successful layout is to place one searcher on every vertex of a graph. Since every searcher only protects their current vertex, we can improve the number of searchers in any optimal successful layout.
\begin{theorem} If $G$ is a non-trivial graph of order $n\geq 2$, then $d(G)\leq n-1$.
\label{4.1.1}
\end{theorem}
\begin{proof} Since $G$ is non-trivial, there is a vertex with degree at least $1$, say $v$. Then a successful layout is obtained by placing one searcher on every vertex of $G$ except $v$.
\end{proof}

Similarly, as every searcher can protect at most two vertices (the one it initially occupies and the one it moves to), we obtain the following result.

\begin{theorem} If $G$ is a graph of order $n$, then $d(G)\geq\lceil\frac{n}{2}\rceil$.
\label{Th_1}
\end{theorem}
Theorems \ref{4.1.1} and \ref{Th_1} are tight. As we will see in Theorems~\ref{CO_3} and \ref{Th_4}, $d(K_n)=n-1$ and $d(P_n)= \lceil \frac{n}{2}\rceil$.

We now give some basic bounds on the deduction number of $G$ based on the degrees of the vertices in $G$.

\begin{theorem} In any connected graph $G$ with order $n\geq3$ and $p\geq 1$ vertices of degree 1, $ d(G)\geq p$.
\label{4.1}
\end{theorem}
\begin{proof} 
In a successful layout on $G$ with $d(G)$ searchers, each vertex $v$ of degree 1 must either be occupied or adjacent to an occupied vertex.  Moreover, since $G$ is connected and of order at least 3, no two vertices of degree 1 can be adjacent. Thus, the number of searchers is at least the number of vertices of degree 1.
\end{proof}
\begin{theorem} For any graph $G$, $d(G)\geq  \delta (G)$.
\label{MinDegree}
\end{theorem}
\begin{proof} Consider a successful layout in the graph $G$ using $d(G)$ searchers, and let $v$ be an occupied vertex whose searchers move in the first round. The number of searchers on $v$ is at least the number of unoccupied vertices adjacent to it. Moreover, by definition each occupied vertex adjacent to $v$ contains at least one searcher.  Hence $d(G)\geq \deg (v) \geq \delta (G)$.
\end{proof}

In the following theorem, we  show that in any graph $G$, there is a successful layout with $d(G)$ searchers in which no vertex is occupied by more than one searcher.  Before proceeding to the theorem, we make the following definition.
\begin{definition}
    A layout on a graph in which each vertex contains at most one searcher is called a {\em standard layout}.
\end{definition}
\begin{theorem} In a graph $G$, there exists a successful standard layout using $d(G)$ searchers. 
\label{4.1.6}
\end{theorem}

\begin{proof}
Suppose that $G$ has a successful layout $L$ using $d(G)$ searchers in which at least one occupied vertex has more than one searcher.  Note that no such vertex can contain motionless searchers, as otherwise we can remove a searcher and still obtain a successful layout.  Among these vertices, choose one, say $v$, whose searchers move at the earliest stage, say stage $s$.  

At stage $s$, the number of searchers on $v$ is equal to the number of unprotected neighbours of $v$, as otherwise $v$ contains a motionless searcher.  So we may choose one of the searchers, $c$, on $v$ and create a new layout $L'$ by moving this searcher to its target, say $u$.

Note that in all stages preceding stage $s$, the movement of searchers in $L'$ is the same as in $L$.  At stage $s$, the searchers on $v$ other than $c$ move to the same targets they would in $L$, and searchers on vertices other than $v$ move as in $L$ unless they would have moved to $u$, in which case they remain motionless.  The searcher $c$ may or may not move from $u$. Thus, at the end of stage $s$, every vertex that is protected in the game with layout $L$ is also protected in the game with layout $L'$.  In each subsequent stage, each searcher who moves in the game with layout $L$ thus either makes the same movement in $L'$, or their target is already protected, in which case they remain motionless.  

We conclude that $L'$ is a successful layout.  By repeating the same process if necessary, we eventually obtain a successful layout with $d(G)$ searchers in which each vertex contains at most one searcher.
\end{proof}

Finally, we conclude this section by noting that certain structural characteristics imply bounds to the deduction number --- namely cut vertices and bridges. 

\begin{theorem}\label{cutvertex}
    Let $G$ be a connected graph with a cut-vertex $u$, and let $G_1, G_2, \ldots, G_t$ be the components of $G-u$.  Then $d(G) \leq 1+\sum_{i=1}^t d(G_i)$.
\end{theorem}

\begin{proof}
    Construct a successful layout on $G$ as follows.  Place searchers on the vertices of the $G_i$ according to optimal successful layouts in each, and place one additional searcher on $u$.
\end{proof}

In fact, we can normally do a little bit better, but the proof requires the introduction of a variation of deduction. For a proof, see \cite{MozhganThesis}.

\begin{theorem} \cite{MozhganThesis} Let $G_1$ and $G_2$ be graphs containing the vertices $v_{1}$ and $v_{2}$ respectively. If $G$ is the graph formed by identifying the vertices $v_1$ and $v_2$ and $G'$ is the graph formed by adding the edge $v_1v_2$, then
$$d(G) \leq d(G_{1})+ d(G_{2})$$
and
$$d(G') \leq d(G_{1})+d(G_{2}).$$
\label{7.1.1}
\end{theorem}

\section{The deduction number on some classes of graphs}
This section determines the deduction number for some classes of graphs, including cycles, paths, complete graphs, complete bipartite graphs and wheels. 

\begin{theorem} If $P_{n}$ is a path of order $n$, then $d(P_{n})=\lceil \frac{n}{2}\rceil$.
\label{CO_3}
\end{theorem}
\begin{proof} It is easy to see that $d(P_{1}) = d(P_{2}) = 1$. For $n \geq 3$, by Theorem~\ref{Th_1}, $d(P_n) \geq \lceil \frac{n}{2}\rceil$.  For the upper bound, let $P_n$ be the path $v_1v_2\cdots v_n$.  If $n$ is even, place a searcher on each vertex of odd index.  If $n$ is odd, place a searcher on $v_1$ and on each vertex of even index.  It is straightforward to verify that in either case we obtain a successful layout with $\lceil\frac{n}{2}\rceil$ searchers.  
\end{proof}
The same result holds for cycles.

\begin{theorem} If $C_{n}$ is a cycle of order $n \geq 3$, then $d(C_{n})=\lceil \frac{n}{2}\rceil$.
\label{4.2.2}
\end{theorem}
\begin{proof} By Theorem~\ref{Th_1}, $d(C_n) \geq \lceil\frac{n}{2}\rceil$. For the upper bound, consider the cycle $C_{n}=(v_1, v_2, \ldots, v_n)$. If $n=3$ or $n=4$, then we place a single searcher on each of two vertices $v_{1}$ and $v_{2}$. In either case, this is a successful layout with $\lceil \frac{n}{2}\rceil$ searchers.

For $n\geq 5$, we consider two cases according to the parity of $n$.  If $n$ is even, place a searcher on each of $v_1, v_2, v_4, v_6, \ldots, v_{n-2}$.  In the first stage, the searchers on $v_1$ and $v_2$ move to $v_{n}$ and $v_3$, respectively.  In subsequent stages, the searcher on $v_{2i}$ ($i \geq 2$) will move to $v_{2i+1}$.  Thus, this layout is successful.

If $n$ is odd, place a searcher on each of $v_1, v_2, v_4, v_6, \ldots, v_{n-3}$ and $v_n$.  In the first stage, the searchers on $v_n$ and $v_2$ move to $v_{n-1}$ and $v_3$, respectively. As in the previous case, in subsequent stages, the searcher on $v_{2i}$ ($i \geq 2$) moves to $v_{2i+1}$, and we again have found a successful layout with $\lceil \frac{n}{2}\rceil$ searchers.
\end{proof}
Consider now the wheel $W_n$ of order $n$.  Theorems~\ref{4.1.1} and~\ref{MinDegree} together imply that $d(W_4)=3$.  For $n>4$, we show in the next theorem that the wheel meets the lower bound from Theorem~\ref{Th_1}.

\begin{theorem} If $W_{n}$ is a wheel of order $n\geq 5$, then $d(W_{n})= \lceil \frac{n}{2}\rceil$.
\label{4.2.3.1}
\end{theorem}

\begin{proof} 

The lower bound follows from Theorem~\ref{Th_1}.  To prove the upper bound, 
label the vertices $v_{0},v_{1},$ $ \ldots,v_{n-1}$, where $v_0$ is adjacent to each other vertex, $v_{n-1}$ is adjacent to $v_0$ and $v_i$ is adjacent to $v_{(i+1)}$ for $i \in \{1, \ldots, n-2\}$. We give a successful layout using $\lceil\frac{n}{2}\rceil$ searchers. 
 
If $n$ is odd, place a searcher on each vertex $v_1$, $v_{n-2}$ and $v_{n-1}$.  For $n \geq 7$, also place a searcher on each vertex $v_2, v_4, v_6, \ldots, v_{n-5}$.  
If $n$ is even, place a searcher on each vertex $v_1, v_3, \ldots, v_{n-5}$, as well as on $v_{n-2}$ and $v_{n-1}$.
In either case, it is straightforward to verify that this layout is successful.

\end{proof}

We next consider the deduction game on complete graphs.  
\begin{theorem} \label{Th_4}If $K_{n}$ is a complete graph of order $n \geq 2$, then $d(K_{n})=n-1$.
\end{theorem}

\begin{proof} This is a direct consequence of Theorems~\ref{4.1.1} and~\ref{MinDegree}.
\end{proof}

The following result considers the deduction number of a graph containing a clique of order $m$.  Combined with Theorem~\ref{Th_4}, it gives a bound on the deduction number of a graph based on its clique number.
\begin{theorem} If $G$ is a graph of order $n$ and $K_{m}$ is a subgraph of $G$ of order $m<n$, then $d(K_{m}) \leq d(G) $.
\label{4.2}
\end{theorem}

\begin{proof} Let $L$ be a successful layout $L$ with $d(G)$ searchers. By Theorem~\ref{4.1.6}, we may assume that $L$ is a standard layout.  We construct a successful layout in a subgraph isomorphic to $K_{m}$ with at most $d(G)$ searchers.  

Let $H$ be a subgraph of $G$ with $m$ mutually adjacent vertices.  Construct a new layout $L'$ on $H$ as follows.  Each vertex of $H$ which is occupied in $L$ remains occupied in $L'$.  Additionally, place a searcher on each vertex of $H$ which is a target of a vertex in $V(G) \setminus V(H)$.  Note that $L'$ is a standard layout on $H$ which contains at most $d(G)$ searchers.

In the new layout $L'$, if all vertices of $H$ are occupied, then $L'$ is successful. Otherwise, let $v_1$ be an unoccupied vertex in $K_{m}$ in $L'$. Then $v_1$ must be unoccupied in $L$, and by construction of $L'$, the searcher who moves to $v$ in the deduction game on $G$ must be on another vertex of $H$. 
Suppose $H$ has a second unoccupied vertex, $v_2$. In order for the layout $L$ in $G$ to have been successful, since $v_1$ and $v_2$ are both protected by searchers on vertices in $H$, there must be a vertex of $H$ with more than one searcher; otherwise, each searcher on a vertex of $H$ has at least two unoccupied adjacent vertices and remains flummoxed throughout the deduction game on $G$. 
Thus, $v_1$ is the only unoccupied vertex in $L'$, so $L'$ is successful. 
\end{proof}

\begin{corollary} If $G$ is a graph, then $d(G) \geq \omega(G) -1$.
\end{corollary}
Note that for an arbitrary (non-complete) subgraph $H$ of $G$, it is not necessarily true that $d(H) \leq d(G)$.  For example, as the next result shows, stars, which are subgraphs of wheels, provide a counterexample.

\begin{theorem} If $S_{n}$ is a star of order $n \geq 2$, then $d(S_{n})=n-1$.
\label{4.2.5.1}
\end{theorem} 

\begin{proof} 
This is a direct consequence of Theorems~\ref{4.1.1} and~\ref{4.1}.
\end{proof}

As the star $S_n$ is isomorphic to $K_{1,n-1}$, it is natural to consider the deduction number of complete bipartite graphs and complete multipartite graphs more generally.
\begin{theorem}\label{thm:multipartite}
Let $G=K_{n_1, n_2, \ldots, n_m}$ be the complete multipartite graph with $m \geq 2$ parts of sizes $n_i$, $1 \leq i \leq m$, and let $N=\sum_{i=1}^m n_i$.  Then
\[
d(G) = \left\{ 
\begin{array}{ll}
N - 1, & \mbox{if } n_1=n_2=\cdots=n_m=1 \mbox{ or } m=2 \mbox{ and } 1 \in \{n_1, n_2\}, \\
N - 2, & \mbox{otherwise}
\end{array}
\right.
\]
\end{theorem}

\begin{proof}
The case that $m=2$ and $1 \in \{n_1,n_2\}$ is just Theorem~\ref{4.2.5.1}.  If $n_1=n_2=\cdots=n_m=1$, then $G \simeq K_N$, and we know $d(K_N)=N-1$ by Theorem~\ref{Th_4}.  
So we henceforth assume there is at least one part of size at least 2 and if $m=2$, then both parts have size at least 2.  

We begin by exhibiting a successful layout using $N-2$ searchers.  First, suppose that there are two parts, say $P_1$ and $P_2$, of size at least 2.  
Choose vertices $u$ in $P_1$ and $v$ in $P_2$, and place a searcher on each vertex other than $u$ and $v$.  This layout is successful since there is at least one searcher in $P_1$, who will move to $v$, and at least one searcher in $P_2$, who will move to $u$.  

Otherwise, exactly one part, say $P_1$, has size at least 2.  Choose a vertex $u$ in $P_1$ and a vertex $v$ in another part, say $P_2$, and place a searcher on each vertex other than $u$ and $v$.  Each searcher in $P_1 \setminus \{u\}$ has exactly one unoccupied neighbour, namely $v$, and so will will move to $v$.  Thus, the searchers on vertices in $V(G) \setminus (P_1 \cup P_2)$, despite having two initially unoccupied neighbours, deduce that they must move to $u$.  Hence the layout is successful.  

We now show that any successful layout $L$ in $G$ must have at least $N-2$ searchers. Let $P_1, P_2, \ldots, P_m$ be the parts of $G=K_{n_1, n_2, \ldots, n_m}$ with $|P_i|=n_i$, and suppose that there are $j_i$ unoccupied vertices in the part $P_i$.  
 Clearly if every vertex is occupied, we have at least $N$ searchers, so we assume that there is at least one unoccupied vertex.  Since $L$ is successful, there must be a searcher who is not flummoxed, so there must be a vertex $u$ in some part $P_k$ with unoccupied neighbours whose searchers can move in the first round, i.e.\ $u$ has at least
\[
\left(\sum_{i=1}^m j_i\right) - j_k
\]
searchers and $j_i \geq 1$ for some $i \neq k$.  If $j_k=0$, then the total number of searchers is thus at least
\[
\sum_{i=1}^m (n_i-j_i) + \left(\sum_{i=1}^m j_i\right) - j_k - 1 = N-1>N-2.
\]
(We subtract 1 to avoid double-counting searchers on $u$.)

Otherwise, $j_k \geq 1$, so there are unoccupied vertices in $P_k$.  There is a vertex $v$ in some part $P_{\ell}$ (with $\ell \neq k$) whose searchers may move to the $j_k$ unoccupied vertices in $P_k$.  This requires that there are at least $j_k$ searchers on $v$ (as this movement may become possible because the searchers on $u$ move to the neighbours of $v$ not in $P_{k}$).  In this case, the total number of searchers is at least
\[
\sum_{i=1}^m (n_i-j_i) + \left(\left(\sum_{i=1}^m j_i\right) - j_k - 1\right) + (j_k-1) = N-2.
\]
\end{proof}
The result of Theorem~\ref{thm:multipartite} is particularly interesting, because it gives the first infinite family of graphs for which we can compute the deduction number but that number does not meet either of the bounds of Theorems~\ref{4.1.1} or \ref{Th_1}.

We conclude this section by considering an elementary result on Cartesian products, which allows us to solve the deduction number of the $n$-hypercube.

\begin{theorem}\label{theorem11} If $G$ and $H$ are graphs, then $$d(G \Box H)\leq \min \{|V(G)| \cdot d(H),|V(H)| \cdot d(G)\}.$$

\end{theorem}
\begin{proof} We assume without loss of generality that $|V(G)| \cdot d(H) \le |V(H)| \cdot d(G)$.

For $u \in V(G)$, let $H_u$ be the copy of vertex set $\{u\} \times V(H)$. Then $G \Box H$ contains $|V(G)|$ copies of the graph $H$. Place $d(H)$ searchers in an optimum successful layout on corresponding vertices in each copy of $H$ in the graph $G \Box H$. Since we repeat the same layout, every occupied vertex in $H_{u_{1}}$ is also occupied in $H_{u_{i}}$, for $1 \leq i \leq |V(G)|$. Thus, every searcher in $H_{u_{i}}$ moves within their copy of $H$ following the movement in the layout on $H$. As a result, $|V(G)| \cdot d(H)$ searchers are sufficient, and the result follows.
\end{proof}

In certain cases, this bound achieves equality.

\begin{theorem} \label{Co8}
If $m \geq 2$ is even, and $G$ is a graph of order $n$, then $d(P_{m} \Box G)=\frac{mn}{2}$.
\end{theorem} 
\begin{proof} By Theorems~\ref{theorem11} and \ref{CO_3}, we have the upper bound, and by Theorem~\ref{Th_1}, we need at least $\frac{mn}{2}$ searchers. Thus, $d(P_{m} \Box G)=\frac{mn}{2}$.
\end{proof}
\begin{corollary} \label{Cor:ProductPn} 
Let $n_1, n_2, \ldots, n_k \geq 2$ be integers. 
 If at least one of $n_1, \ldots, n_k$ is even, 
then $d(P_{n_{1}} \Box P_{n_{2}}\Box\cdots \Box P_{n_{k}})=\frac{n_{1}\cdot n_{2}\cdots n_{k}}{2}$.
\end{corollary}
\begin{proof} Without loss of generality, suppose $G$ is the Cartesian product \linebreak $P_{n_{2}}\Box \cdots \Box P_{n_{k}}$ and $n_{1}$ is an even number. Then, by Theorem~\ref{Co8} the equality holds.
\end{proof}
Recall that the $n$-hypercube $Q_{n}$ is the Cartesian product of $n$ copies of $P_{2}$. 
\begin{corollary} If $n \geq 1$, then $d(Q_{n})=2^{n-1}$.
\end{corollary}
 \begin{proof} For $n=1$, $d(Q_{1})=d(P_2)=1$. Applying Corollary~\ref{Cor:ProductPn} for $n\ge 2$, the result follows. 
\end{proof}

\section{The deduction number on trees}
 
In this section, we consider the deduction number of trees, and give an algorithm which finds an optimal successful layout on a tree.  We begin with a lemma regarding the structure of successful layouts on trees.  Here, a {\em stem} refers to a vertex adjacent to a leaf.
\begin{lemma} Let $T$ be a tree of order $n \geq 3$, and let $L$ be a layout in which a stem $s$ is occupied.  There exists a successful layout $L'$ with the same number of searchers as $L$, in which $s$ is unoccupied.

\label{5.1.1}
\end{lemma}
\begin{proof}
Given the layout $L$, we construct $L'$ by moving the searchers on $s$ to their target vertices.
If there are more searchers on $s$ than target vertices, we move these ``excess searchers'' to leaves adjacent to $s$. 

Let $G$ denote the deduction game played with layout $L$, and $G'$ the game with layout $L'$.  
Note that in $L'$, 
each
leaf adjacent to $s$ must be occupied. Each searcher positioned on the leaves adjacent to $s$ in $L'$ has one unoccupied adjacent vertex, $s$; therefore, each such searcher moves to $s$ in the first stage of $G'$.

We now show that the vertices other than $s$ become protected in $G'$ by considering stages in which searchers move.  
First consider searchers who move in the first stage of $G$.  If the searchers on vertex $v$ do not move to the same target in the first stage of $G'$, then $v$ must be adjacent to $s$.  Since $s$ now becomes occupied in the first stage, the searchers on $v$ can move to their targets in the second stage (if these targets have not already become protected in the first stage).

Next, inductively assume that all vertices that were protected by stage $i$ in the $G$ become protected by stage $i+1$ in $G'$.  Consider a vertex $u$ which becomes protected in stage $i+1$ in $G$, and suppose that a searcher moves from $w$ to $u$ in this stage.  If $u$ becomes protected by stage $i+1$ in $G'$, we are done.  Otherwise, all the neighbours of $w$ which were protected by the end of stage $i$ in $G$ become protected by the end of stage $i+1$ in $G'$, meaning that a searcher on $w$ may move to $u$ in stage $i+2$.
\end{proof}
Note that if the layout $L$ in the statement of Lemma~\ref{5.1.1} is standard and contains $d(G)$ searchers, then the layout $L'$ produced will also be standard.

Now, we introduce 
an algorithm to determine the deduction number of trees, the 
{\em pruning algorithm}.  
The basic idea is to place searchers on leaves to protect their unoccupied neighbours, 
remove the leaves and their target vertices, and repeat this procedure
until no vertex remains. We call the number of searchers in the layout constructed in this way the {\it pruning number} of $T$, denoted by $p(T)$. 

\begin{algorithm}[H]
\label{alg1}

\SetKwInOut{Input}{input}\SetKwInOut{Output}{output}

\Input{A tree $T$}

\Output {A set of occupied vertices forming a layout on $T$}

\BlankLine

initialise $S$ to $\emptyset$

\Repeat{no vertices remain}{

\ForAll{connected components $C$ of $T$}{

\uIf{$|V(C)| \leq 2$}{

Add one vertex of $C$ to $S$

}

\Else{

Add each leaf of $C$ to $S$

}

Delete leaves, stems and isolated vertices from $T$

}

}

\Return{$S$}

\caption{Pruning Algorithm}
\end{algorithm}

\bigskip

It is easy to see that the pruning algorithm produces a successful layout for a tree.  As it adds the minimum number of searchers required to protect the remaining vertices at each iteration, it is reasonable to suspect that it constructs an optimal layout. We now prove that this is indeed the case.
\begin{theorem} In a tree $T$, $p(T)=d(T)$. 
\label{5.1.5}
\end{theorem}
\begin{proof} 
We use induction on the number of vertices, $n$. For trees of order $n=1$ and $n=2$, the result holds, since in this case, the deduction and pruning number are both 1. Let $T$ be a tree of order $n\geq3$, and suppose that in any tree $T'$ with $n-1$ or fewer vertices, 
$p(T')=d(T')$. 

Now consider an optimal successful layout $L$ on $T$. 
We will convert $L$ 
to the layout given by the pruning algorithm. By Theorem~\ref{4.1.6}, we may assume that $L$ is a standard layout.

Let $s$ be a stem in $T$.  Using Lemma~\ref{5.1.1} if necessary, we may assume that $s$ is unoccupied in $L$.  Since $L$ is a successful layout, each leaf adjacent to $s$ must be occupied.  Choose one such leaf, $\ell$, and delete $s$ and $\ell$.
What remains is a collection of components, each of  which is a tree 
of order less than $n-1$. In each of these components, by the induction hypothesis, the pruning algorithm produces an optimal layout; thus we 
can replace the existing layout on each component with the one given by the pruning algorithm without changing the number of searchers. 
Adding back $\ell$ and $s$, with a searcher on $\ell$, we obtain a layout $L_p$ on $T$, which is precisely the layout that would be produced by the pruning algorithm.
Note that the number of searchers in $L_p$ 
is still $d(T)$.  
Therefore, $p(T) = d(T)$.
\end{proof}


\section{Discussion and future work}

This paper has introduced several ideas that would benefit from further research. Consider the idea of a terminal layout. Given a layout $L$, the most basic question we can ask must be: is the terminal layout $L^\ast$ itself always successful? Assuming it is, can we characterize {\em reversible} layouts? That is, can we characterize those layouts such that $L = L^{\ast\ast}$? Moreover, for every layout $L$, the sequence $L$, $L^\ast$, $L^{\ast\ast}$, $\ldots$ seems to have two limit points --- that is, eventually, this sequence appears to be attracted to a pair of reversible layouts! However, without first characterizing reversible layouts, proving this latter result seems like a pipe dream.

We have also introduced several ways to bound the deduction number of a graph constructed from smaller graphs, such as identifying vertices, adding bridges between disconnected graphs, and of course, considering the Cartesian product of two graphs. For the latter case in particular, it is clear that the bound of Theorem~\ref{theorem11} is far from tight in most cases, even though it allowed us to exactly compute the deduction number of hypercubes. For example, consider $G \Box H$, when there is a vertex that contains a motionless searcher in $G$. Then, in every copy of $G$ in the product, that searcher is also motionless. However, each of those searchers is on the same copy of $H$. Instead, we could attempt to use some subset of them to clear that copy of $H$. Unfortunately, how to interlace the two successful layouts into a successful layout in the product is difficult to characterize in general. 

There are also several other elementary directions that are natural to consider for deduction. What is the algorithmic complexity of calculating the deduction number? Is there a structural characterization, as with copwin graphs? Since we understand deduction on trees and cliques, what can we say about the deduction number of chordal graphs in general? Some of these questions will be explored in~\cite{BDOXY}.

\section{Acknowledgements}

Authors Burgess and Dyer acknowledge grant support from NSERC Discovery Grants RGPIN-2019-04328 and RGPIN-2021-03064, respectively.

\bibliographystyle{abbrv}
\bibliography{mybibliography}

\end{document}